\documentclass[11pt]{article}
\usepackage{graphics, amsfonts, amsmath, amsthm, amstext, amssymb, amscd, color, rotating, makeidx}
\usepackage{calligra,mathrsfs} 

\newcommand{\Aut}{\mbox{Aut}\,}
\newcommand{\rad}{\mbox{rad}\,}

\newcommand{\id}{\mbox{id}\,}

\newcommand{\adj}{\mbox{adj}\,}

\newcommand{\Ric}{\mbox{Ric}\,}

\newcommand{\Supp}{\mbox{Supp}\,}

\newcommand{\im}{\text{Im\,}}
\newcommand{\barD}{\bar{\partial}}
\newcommand{\ddbar}{\partial \bar{\partial}}
\newcommand{\pd}{\partial}

\newtheorem{theorem}{Theorem}[section]
\newtheorem{definition}[theorem]{Definition}
\newtheorem{lemma}[theorem]{Lemma}
\newtheorem{proposition}[theorem]{Proposition}

\newtheorem{corollary}[theorem]{Corollary}

\newenvironment{remark}{\noindent{\bf Remark:}}{\vspace{3mm}}

\begin{document}

\title{On Effective Existence of Symmetric Differentials of Complex Hyperbolic Space Forms
}

\author{Kwok-Kin Wong\footnote{Department of Mathematics, The University of Hong Kong. Email: kkwongm@gmail.com}}
              
\date{}

\maketitle

\begin{abstract}
For a noncompact complex hyperbolic space form of finite volume $X=\mathbb{B}^n/\Gamma$, 
we consider the problem of producing symmetric differentials vanishing at infinity on the Mumford compactification $\overline{X}$ of $X$ similar to the case of producing cusp forms on hyperbolic Riemann surfaces. We introduce a natural geometric measurement which measures the size of the infinity $\overline{X}-X$ called {\bf canonical radius} of a cusp of $\Gamma$. The main result in the article is that there is a constant $r^*=r^*(n)$ depending only on the dimension, so that if the canonical radii of all cusps of $\Gamma$ are larger than $r^*$, then there exist symmetric differentials of $\overline{X}$ vanishing at infinity. As a corollary, we show that the cotangent bundle $T_{\overline{X}}$ is ample modulo the infinity if moreover the injectivity radius in the interior of $\overline{X}$ is larger than some constant $d^*=d^*(n)$ which depends only on the dimension.

\end{abstract}

\section{Introduction}
\label{intro}
Let $\Omega$ be an irreducible bounded symmetric domain and $\Gamma\subset \Aut_0(\Omega)$ be a torsion-free lattice so that $X=\Omega/\Gamma$ is a noncompact finite volume quotient with respect to the canonical K\"ahler-Einstein metric. Our main interest will be certain hyperbolic properties of the compactifications of $X$.  We say that a complex space $M$ is (Brody) hyperbolic if there exists no non-constant holomorphic map $f:\mathbb{C}\rightarrow M$. When $M$ is compact, being hyperbolic is equivalent  to that the Kobayashi pseudo distance function on $M$ is actually a distance function.

It is well-known, however, that $\overline{X}$ is in general not hyperbolic. For example, there is (noncompact) modular curve $X_k=\mathbb{H}/\Gamma_0(k)$ for some small $k\in \mathbb{N}$, such that $\overline{X_k}$ is $\mathbb{P}^1$ or elliptic curve. Nonetheless, suppose $X_k=\mathbb{H}/\Gamma_0(k)$ is a modular curve so that $\overline{X_k}$ is not hyperbolic. There is a sufficiently large $N\geq k$ such that $\overline{X_N}$ is hyperbolic. In this case, we say that $\overline{X_k} $ is hyperbolic up to cover. Such an idea has been generalized by Nadel to arithmetic quotients of bounded symmetric domain:
\begin{theorem}[cf. Theorem 0.1, Nadel \cite{Nadel1989}]
Let $\Omega$ be a bounded symmetric domain, $\Gamma\subset \Aut(\Omega)$ be an arithmetic lattice and $\overline{\Omega/\Gamma}$ be any compactification of $\Omega/\Gamma$. Then there  is a cover $\Omega/\Gamma' \rightarrow \Omega/\Gamma$ such that the image of any holomorphic curve $f:\mathbb{C}\rightarrow \overline{\Omega/\Gamma'}$ lies on $\overline{\Omega/\Gamma'}-\Omega/\Gamma'$.
\label{nadel}
\end{theorem}
Theorem \ref{nadel} has an important consequence on hyperbolicity of $\Omega/\Gamma$. For instance, if one takes the Satake-Baily-Borel compactification $X'$ of $X=\Omega/\Gamma$, then one concludes that $X'$ is hyperbolic up to cover since $X$ is compactified again by quotients of bounded symmetric domains. The hyperbolicity of the Mumford compactification $\overline{X}$, however, is still not known by our knowledge.

We are going to look at the possibly simplest situation in higher dimensions, where $\Omega=\mathbb{B}^n$ is the complex unit ball of dimension $n\geq 2$. Let $X=\mathbb{B}^n/\Gamma$ be quotient of the complex unit ball by a torsion-free lattice $\Gamma\subset \Aut(\mathbb{B}^n)$ of finite volume with respect to the canonical K\"ahler-Einstein metric. The noncompact manifold $X$ admits toroidal compactification $\overline{X}$ by \cite{AMRT} in case $\Gamma$ is arithmetic and by \cite{Mok2012} for non-arithmetic $\Gamma$. It is shown that $\overline{X}$ possesses many hyperbolic properties (see for example, \cite{BT2015a}, \cite{Cad2016}, \cite{DCDC2015}, \cite{Wang2015}).

It is well-known that a projective manifold with ample cotangent bundle is hyperbolic. However, the converse is false. For example, take a product of Riemann surfaces of genus $\geq 2$. For the situation $\overline{X}$ is the Mumford compactification of $X=\mathbb{B}^n/\Gamma$, we know that it is in general not hyperbolic and hence has a non-ample cotangent bundle. We will study a weaker notion of ampleness of cotangent bundle, namely, the {\bf ampleness modulo a divisor}. We say that a vector bundle $E\rightarrow M$ over a complex space is ample modulo a divisor $D\subset M$ if the global sections of the tautological line bundle $\mathcal{O}_{\mathbb{P}E^*}(1)$
\footnote{The notation $\mathbb{P}E$ is the projectivization of the bundle $E$ instead of the dual of $E$, which is different from literatures in algebraic geometry.}
induces an embedding of $\mathbb{P}E^*-\pi^{-1}(D)$ into some projective space, where $\pi:\mathbb{P}E^*\rightarrow M$ is the natural projection. To us, Nadel's theorem is a consequence of the existence of pluricanonical sections vanishing at infinity. In the same spirit, we are going to show that 
\begin{theorem}[Main Theorem]
Let $X=\mathbb{B}^n/\Gamma$ be a quotient of the complex unit ball $\mathbb{B}^n\subset \mathbb{C}^n$ ($n\geq 2$) by a torsion-free lattice $\Gamma\subset Aut(\mathbb{B}^n)$, which is not necessarily arithmetic, of finite volume with respect to the canonical K\"ahler-Einstein metric $g_X$ that induced from the normalized Bergman metric $g$ on $\mathbb{B}^n$ of constant holomorphic sectional curvature $-2$. 
Write $\overline{X}$ to be the Mumford compactification of $X$.
Let $b_1,\dots,b_k\in \partial \mathbb{B}^n$ be the cusps of $\Gamma$ with canonical radii $r_1\geq \dots \geq r_k$ so that $\Omega:=\coprod_{j=1}^{k} \Omega_{b_j}$ is a disjoint union of neighbourhoods in $\overline{X}$ corresponding to the cusps.
Denote by $\Omega'$ the complement of $\Omega$
and by $d(\Omega')$ the injectivity radius of $\Omega'$ with respect to $g$.  

Then there are constants $r^*=r^*(n)$ and $d^*=d^*(n)$  depending only on the dimension $n$ such that if $d(\Omega')\geq d^*$ and $r_1\geq \dots \geq r_k\geq r^*$, the cotangent bundle $T^*_{\overline{X}}$ is ample modulo $D=\overline{X}-X$.
\label{main}
\end{theorem}
We say that $\{X_k\}_{k=0}^{\infty}$ is a tower of coverings of $X$, if $X_0=X$ and for each $k\geq 0$, $X_{k+1}\rightarrow X_k$ is a finite covering such that $\pi_1(X_{k})\subset \pi_1(X)$ is a normal subgroup and such that $\cap_{i=0}^{\infty} \pi_1(X_k)=\{1\}$. For example, it is well-known that any arithmetic quotient of $\mathbb{B}^n$ has a tower of coverings. Since both injectivity radii of the interior part and the canonical radii corresponding to cusp neighbourhoods increase upon lifting along the tower of coverings, a direct consequence of Theorem \ref{main} is
\begin{corollary}
Suppose $\{X_k\}_{k=0}^{\infty}$ is a tower of coverings of $X$. Then there is $k\geq 0$ such that the cotangent bundle $T^*_{\overline{X_k}}$ is ample modulo $\overline{X_k}-X_k$. Hence if $\Gamma'$ is any finite index subgroup of $\pi_1(X_k)$, then the cotangent bundle of $\overline{X'}$, where $X'=\mathbb{B}^n/\Gamma'$, is ample modulo $\overline{X'}-X'$.
\label{main2cor}
\end{corollary}

The proof of Theorem \ref{main} is done by methods of $L^2$-estimates. From our point of view, the proof boils down to the construction of holomorphic symmetric differentials vanishing at infinity, which can also be thought of as a generalization of cusps form in modular curves. As inspired by the works of Mumford \cite{Mum1977}, Nadel \cite{Nadel1989} and Bakker-Tsimerman \cite{BT2015a}, an important element for the existence of the desired symmetric differentials is to make sure the size of the lattice ``$\Gamma$ is sufficiently small" or the infinity ``$D$ is not too big". Hence we introduce the notion of {\bf canonical radius} of a cusp, which measures the size of the infinity corresponding to a cusp of $\Gamma$. The canonical radius is probably a well-known quantity among experts in hyperbolic geometry (see for example, \cite{Par1998}). 

Our method relies on certain asymptotic properties of the canonical metric $g_X$ on $X$ and the curvature $\Theta$. On the other hand, in the proof of the ampleness modulo infinity,  the injectivity radius will also play an essential role.
The key point is, in the verification of the basic estimates required for $L^2$-estimates, the curvature term is positive and of growth $1/\delta^2$ as $\delta\rightarrow 0$, where $\delta$ is the distance to the boundary. The only negativity coming from the weight function is of growth $1/\delta$ as $\delta\rightarrow 0$. Thus the basic estimates are verified if the canonical radii are large so that the support where the negativity occurs lie in region close enough to the boundary.
Another important ingredient is the construction of locally holomorphic symmetric differentials vanishing along the infinity by applying a technique of Grauert (see Proposition \ref{locdiff}).

Recently, there seems to be an interest in producing symmetric differentials for Hermitian locally symmetric spaces in relation to the fundamental group representation and Hodge theory, see for example \cite{Kli2013}, \cite{BK2013}. Especially in \cite{Kli2013} Theorem 1.6 and Theorem 1.26, where the effectiveness plays an important role. We hope that in the near future, such relations can be further explored in the context of complex analytic geometry.

The organization of the rest of the article is as follows. 
In \S\ref{mum}, we recall briefly the Mumford compactification of $X=\mathbb{B}^n/\Gamma$.
In \S\ref{metric}, several asymptotic properties of the Bergman metric and the induced canonical metric will be illustrated.
In \S\ref{constr}, the construction of symmetric differential vanishing at infinity will be done.
In \S\ref{amp}, the ampleness modulo infinity of the cotangent of $\overline{X}$ will be shown, as well as Corollary \ref{main2cor}.

\section{Descriptions of the Mumford Compactification of $X=\mathbb{B}^n/\Gamma$}\label{mum}
We first recall the Mumford compactification of $X=\mathbb{B}^n/\Gamma$. Our presentation follows \cite{Mok2012}.

Fix a cusp $b\in \partial \mathbb{B}^n$ of $\Gamma$. By composing a unitary transformation and the inverse Cayley transformation, we may let $b=q_\infty$ be the distinguish infinity of the Siegel upper half space $S$. 
Write $\Gamma_\infty=\Gamma_{q_\infty}=\Gamma\cap P_\infty$ (where $P_\infty:=P_{q_\infty}\subset \Aut(\mathbb{B}^n)$ is the isotropy subgroup of a point $q_\infty$). Let $U_\infty=\rad P_\infty$ be the unipotent radical of $P_\infty$. Note that $\Gamma\cap [U_\infty,U_\infty]$ is one dimensional and is generated by some $\tau>0$.
 
We let $\Omega^{(u)}_b\subset \overline{X}$ to be a horoball neighbourhood of the cusp $b$ of height $u>0$, which can be described more explicitly as follows.
For each $u>0$, let 
\[
S^{(u)}:=\{(z',z_n)\in \mathbb{C}^n \mid \im z_n> |z'|^2+u \}.
\]
Thus $u_1>u_2$ implies $S^{(u_1)}\subset S^{(u_2)}\subset S^{(0)}=:S$.
Consider the quotient of $S$ by $\Gamma\cap [U_\infty,U_\infty]$ given by the holomorphic covering map
\[
\begin{array}{cccc}
\Psi: &\mathbb{C}^{n-1}\times \mathbb{C} &\rightarrow & \mathbb{C}^{n-1}\times \mathbb{C}^* \\
& (z',z_n) &\mapsto & (z',e^{ \frac{ 2\pi \sqrt{-1}}{\tau} z_n}):=(w',w_n),
\end{array}
\]
where $\tau\in \mathbb{R}$ corresponds to the generator of $\Gamma\cap [U_\infty,U_\infty]$.
In the Mumford compactification, the cusp is compactified by taking interior closure of $\mathbb{C}^{n-1}\times \Delta^*\subset \mathbb{C}^{n-1}\times \Delta$.

Let $G:=\Psi(S), G^{(u)}:=\Psi(S^{(u)})$. For each $u>0$, the canonical projection $G^{(u)}\rightarrow \mathbb{C}^{n-1}$ realizes $G^{(u)}$ as the total space of a puncture disk bundle over $\mathbb{C}^{n-1}$.
Let 
\[
\begin{array}{c}
\hat{G}:=\{(w',w_n)\in \mathbb{C}^{n} \mid |w_n|^2< e^{\frac{-4\pi}{\tau}|w'|^2}\}, \\
\hat{G}^{(u)}:=\{(w',w_n)\in \mathbb{C}^{n} \mid |w_n|^2< e^{\frac{-4\pi}{\tau}u}\cdot e^{\frac{-4\pi}{\tau}|w'|^2}\}.
\end{array}
\]
The sets $\hat{G}$ and $\hat{G}^{(u)}$ are obtained by glueing the zero section $\mathbb{C}^{n-1}\times \{0\}$ to $G$ and $G^{(u)}$ respectively. 

One may show that $G=\Phi(S)\cong S/(\Gamma \cap[U_b,U_b])$. Hence there is a group homomorphism $\varphi:\Gamma\cap U_b\rightarrow \Aut(G)$ such that $\Phi\circ \nu=\varphi(\nu)\circ \Psi$ for any $\Gamma\cap U_b$. This gives rise to the identifications $S/(\Gamma\cap U_b)\cong G/\varphi(\Gamma\cap U_b))$ and $G^{(u)}\cong S^{(u)}/(\Gamma \cap [U_b,U_b])$.

The boundary divisor $T_b$ corresponding to the cusp $b$ is actually the compact complex torus $(\mathbb{C}^{n-1}\times \{0\})/ \Lambda_b$ for some lattice of translations $\Lambda_b$.
Moreover, 
\[
\Omega_b^{(u)}:=\hat{G}^{(u)}/\varphi(\Gamma\cap U_b)\supset G^{(u)}/\varphi(\Gamma\cap U_b)=\Omega_b^{(u)}-T_b
\]
and $\Omega_b^{(u)}$ contains $T_b$. We have a natural open embedding
\[
\Omega^{(u)}_b-T_b\hookrightarrow X=S/\Gamma.
\]
for $u$ sufficiently large.
\begin{definition}
We call the smallest $u_b>0$ so that $\Omega^{(u_b)}_b-T_b\hookrightarrow X$ is an open embedding the {\bf canonical height} of $b$.
The Euclidean radius $r_b=e^{-\frac{2\pi}{\tau}u_b}$ of the disk bundle $\Omega^{(u_b)}_b\rightarrow T_b$ would be called the {\bf canonical radius} of $b$.
\label{depth}
\end{definition}
For any large enough $u>0$, we may identify $\Omega^{(u)}$ with a tubular neighbourhood of the zero section of a negative holomorphic line bundle $L\rightarrow T_b$, whose construction goes as follows. Consider the trivial line bundle
\[
\mathbb{C}^{n-1}\times \mathbb{C} \rightarrow \mathbb{C}^{n-1}
\]
given by the coordinate projection. For each $w=(w',w_n)\in \mathbb{C}^{n-1}\times \mathbb{C}$, we define a Hermitian metric
\[
\mu(w,w):= e^{\frac{4\pi}{\tau}|w'|^2}\cdot |w_n|^2.
\]
Then 
\[
\Theta(\mathbb{C}^{n-1}\times \mathbb{C}, \mu)=-\sqrt{-1}\partial\barD \log \mu=-\frac{4\pi}{\tau}\sqrt{-1}\partial\barD |w'|^2<0.
\]
We may write
\[
\hat{G}^{(u)}=\{w\in \mathbb{C}^{n} \mid \mu(w,w)< e^{\frac{-4\pi}{\tau}u}\}.
\]
So $\hat{G}^{(u)}$ is a level set under the metric $\mu$.
Restrict the trivial line bundle $\mathbb{C}^{n-1}\times \mathbb{C}\rightarrow \mathbb{C}^{n-1}$ to $\hat{G}^{(u)}$ and take quotients, we obtain a line bundle $L\rightarrow T_b$ with the induced Hermitian metric $\overline{\mu}$ of negative curvature (hence showing that $T_b$ is an abelian variety). The tubular neighbourhood $\Omega^{(u)}_b=\hat{G}^{(u)}/\varphi(\Gamma\cap U_b)$ 
is just the set of vectors on $L$ with $\overline{\mu}$-length $< e^{\frac{-2\pi}{\tau}u}$. Under the metric $\overline{\mu}$, we may view $\Omega^{(u)}_b$ as the total space of a disk bundle with each fibre isometric to the disk $D_r$ of radius $0<r=e^{\frac{-2\pi}{\tau}u}<1$ over $T_b$. Then $\Omega^{(u)}_b-T_b$ is a family of puncture disks. 

We summarise some important features of the Mumford compactification as follows:
\begin{theorem}[cf. Mok \cite{Mok2012}]
Let $X=\mathbb{B}^n/\Gamma$ be a noncompact finite volume quotient by a torsion-free lattice.
\begin{enumerate}
\item Set theoretically, the Mumford compactification is a disjoint union
\[
\overline{X}=X\coprod (\coprod_b T_b),
\]
where $T_b$ is an abelian variety and the disjoint union $\coprod_b T_b$ is taken over the finitely many cusps of $\Gamma$.
\item Each $T_b$ can be blown-down to a cusp $b$,  so that the image of the blown-down of $\overline{X}$ is a normal projective variety $X'$, which is the Satake-Baily-Borel compactification of $X$. 
\item For each $T_b$, it admits a system of fundamental neighbourhoods $\{\Omega_b^{(u)}\}_{u\geq u_b}$ in $\overline{X}$. For some large enough $u_0\geq u_b$ (for every cusp $b$), the fundamental neighbourhoods of two different $T_b$'s can be taken to be disjoint. As a complex manifold, we have
\[
\overline{X}=X\coprod (\coprod_b \Omega_b^{(u)})/\sim, \quad \forall u\geq u_0,
\]
where $\sim$ identifies the same point that occurs in both $X$ and $\Omega_b^{(u)}$.
\item Each fundamental neighbourhood $\Omega^{(u)}_{b}$ is a level set of a norm induced by a Hermitian metric $\overline{\mu}$ of a negative line bundle $L$ over $T_b$. One may identify $\Omega^{(u)}_{b}$ as an open subset of the total space of the line bundle $L$, which is then identified to the normal bundle $\mathcal{N}_{T_b|\overline{X}}$ of $T_b$ in $\overline{X}$.
\item Each $\Omega^{(u)}_b\rightarrow T_b$ has a structure of disk bundle. Under the norm $\|\cdot \|$ induced by $\overline{\mu}$, each disk is of radius $r=e^{-\frac{2\pi}{\tau}u}$, where $\tau$ is the generator of $\Gamma\cap [U_b,U_b]$ (where $U_b= \rad P_b$ is the unipotent radical of $P_b$ and $P_b\subset \Aut(\mathbb{B}^n)$ is the isotropy subgroup of $b$).
\end{enumerate}
\end{theorem}

\section{Asymptotic Behaviour of Canonical Metric}\label{metric}
\subsection{The Bergman Metric on $\mathbb{B}^n$}
The normalized Bergman metric $g=(g_{i\overline{j}})$ on $\mathbb{B}^n$ is a K\"ahler-Einstein metric with the K\"ahler form
\[
\begin{array}{rcl}
\omega
&=& \sqrt{-1}\pd\barD(- \log (1-|z|^2)) \\
&=& \displaystyle \frac{\sqrt{-1}}{(1-|z|^2)^2}\sum\limits_{i,j=1}^n \left[(1-|z|^2)\delta_{ij}+\overline{z}_iz_j \right]dz_i\wedge d\overline{z}_j,
\end{array}
\]
which is of constant holomorphic sectional curvature $-2$. To understand the asymptotic behaviour of $g$ as $z\rightarrow \partial \mathbb{B}^n$, by applying automorphisms, one may consider without loss of generality the points $z=(t,0,\dots,0)$, where $0<t<1$. We immediate see by direct calculation that:
\begin{lemma}[cf. Lemma 5, Mok \cite{Mok2012}]
Let $g=(g_{i\overline{j}})$ be the normalized Bergman metric on $\mathbb{B}^n$. Then 
\begin{equation}
 g_{i\overline{j}}\sim 
\begin{cases}
\frac{1}{\delta^2}, \quad i=j=1  \\
\frac{1}{\delta}, \quad i=j\geq 2 \\
0, \quad \text{ otherwise}
\end{cases}
\label{bergball}
\end{equation}
as $t\rightarrow 1$,  i.e., $\delta(t)=1-t\rightarrow 0$.
\end{lemma}

\subsection{Canonical Metric on $X$ Near Infinity}
We now focus on points of $X$ which lie in a cusp neighbourhood $\Omega_b^{(u)}\supset T_b$. 
Let $z_1,\dots,z_n$ denote the coordinates of the upper half space presentation $S$ of $\mathbb{B}^n$ with K\"ahler form
\[
\omega_S=\sqrt{-1}\ddbar (-\log(\im z_n-|z'|^2)).
\]
Fix any $u\geq u_b$, where $u_b$ is the canonical height of the cusp $b\in \partial \mathbb{B}^n$. Let $w_1,\dots,w_n$ to be the local coordinates in $\Omega_b^{(u)}=\hat{G}^{(u)}/\varphi(\Gamma\cap U_b) $ and the divisor $T_b$ be defined by $\{w_n=0\}$. We may view $(w_1,\dots,w_n)\in \hat{G}^{(u)}$.
At a point $(w_1,\dots,w_n)=(w',w_n)\in \hat{G}^{(u)}$, we have
\[
w_1=z_1,\dots,w_{n-1}=z_{n-1},w_n=e^{\frac{2 \pi i}{\tau}\cdot z_n},
\]
where $\tau>0$ corresponds to the generator of some real one-parameter group.
Thus
\[
dw_1=dz_1,\dots, dw_{n-1}=dz_{n-1}, \frac{dw_n}{w_n}=\frac{2\pi \sqrt{-1}}{\tau} dz_n.
\]
Write $\|\cdot \|$ as the norm on $\Omega^{(u)}_b=\hat{G}^{(u)}/\varphi(\Gamma\cap U_b)$ given by the metric $\overline{\mu}$ (induced from the Hermitian metric of $L\rightarrow T_b$):
\begin{equation}
\| w \|:=e^{\frac{2\pi}{\tau}|w'|^2}\cdot |w_n|.
\label{normg}
\end{equation}
We have
\[
\log \|w\| =\frac{2\pi}{\tau} |w'|^2 +\log |w_n|=\frac{2\pi}{\tau}(|z'|^2-\im z_n).
\]
Then
\begin{equation}
\begin{array}{rcl}
\omega_{\Omega_b^{(u)}} &=& \sqrt{-1} \ddbar (-\log(-\log \|w\|)\\
&=& \displaystyle \frac{\sqrt{-1}\ddbar \log \|w\|}{-\log \|w\|}+ \frac{\sqrt{-1} \partial{(-\log \|w\|)}\wedge \barD{(-\log\|w\|)}}{(-\log\|w\|)^2} \\
&=& \displaystyle \frac{1+\log\|w\|}{\|w\|^2(-\log\|w\|)^2} \sqrt{-1}\partial\|w\|\wedge \barD\|w\|+\frac{1}{\|w\|(-\log \|w\|)} \sqrt{-1}\ddbar \|w\| \\
\end{array}
\label{formO}
\end{equation}
We may express the Hessian $\sqrt{-1}\ddbar\|w\|$ in terms of gradient square $\sqrt{-1}\partial \|w\|\wedge \barD\|w\|$ by noting that
\[
\begin{array}{rcl}
\ddbar\|w\|&=&\frac{2\pi}{\tau}\|w\|\ddbar|w'|^2+\frac{\partial \|w\|\wedge \barD\|w\|}{\|w\|}.
\end{array}
\]
Thus
\begin{eqnarray*}
\omega_{\Omega_b^{(u)}}&=& \displaystyle \frac{1}{\|w\|^2(-\log\|w\|)^2} \sqrt{-1}\partial\|w\|\wedge \barD\|w\|+\frac{2\pi}{\tau}\frac{1}{(-\log \|w\|)} \sqrt{-1}\ddbar|w'|^2 .
\end{eqnarray*}
For any $0<\|w\|\leq 1$, if follows from elementary calculus that $0\leq 1-\|w\|\leq -\log \|w\|$. Thus
\[
\frac{1}{(-\log\|w\|)^2}\leq \frac{1}{(1-\|w\|)^2}.
\]
On the other hand, observe that
\[
(\frac{1}{\|w\|^2}\partial\|w\|\wedge \barD\|w\|)^2=0,\quad (\ddbar|w'|^2)^n=0.
\]
We have
\[
\omega_{\Omega_b^{(u)}}^n
=\displaystyle (\frac{2\pi}{\tau})^{n-1}\frac{1}{\|w\|^2(-\log\|w\|)^{n+1}}(\sqrt{-1})^n\partial\|w\|\wedge \barD\|w\|\wedge (\ddbar|w'|^2)^{n-1} .
\]
Hence
\begin{lemma}[cf. Proposition 1, Mok \cite{Mok2012}]
Let $g_X=\{g_{\alpha\overline{\beta}}\}$ be the canonical K\"ahler-Einstein metric on $X$ that is induced by the normalized Bergman metric on $\mathbb{B}^n$. Let $\Omega_b\subset \{0\leq \|w\|\leq 1\}$ be a fundamental neighbourhood of the infinity $T_b=\{w_n=0\}$. Let $\delta(w):=1-\|w\|$ be the distance from a point $w\in \Omega_b$ to $T_b$.
\begin{enumerate}
\item When $\|w\|\rightarrow 0$,
\begin{equation}
\frac{C_1}{\|w\|^2(-\log \|w\|)^{n+1}}dV \leq d V_{g_X} \leq \frac{C_2}{\|w\|^2(-\log \|w\|)^{n+1}} dV,
\end{equation}
where $C_1,C_2>0$ are constants and $dV=\sqrt{-1}\ddbar |w|^2$ is the local Euclidean volume form near infinity.
\item When $\|w\|\rightarrow 1$, $g_{\alpha\overline{\beta}}$ is of growth $1/ \delta^2$.
\end{enumerate}
\end{lemma}

\begin{remark}
The same two sided asymptotic behaviour of $g$ occurs in the Poincar\'e metric of the punctured complex unit disk of dimension $1$, where the norm $\|w\|$ is simply the Euclidean distance.
\end{remark}

For the dual metric $g^*$ on the cotangent bundle $T^*_{X}$, note that $g^*=((g_X)^{-1})^t=\frac{1}{\det g_X} (\adj g_X)^t$.
The norm $\|\cdot \|_{g^*}$ for a holomorphic section $s=s_1dw_1+\dots+s_n dw_n$ of $T^*_{\overline{X}}|_{\Omega_b^{(N)}}$ induced by $g^*$ has the asymptotic behaviour given by $\frac{1}{\det g_X}$ when $\|w\|\rightarrow 0$, i.e.,
\[
C_1 |s_n|^2 \|w\|^2(-\log\|w\|)^{n+1}\leq \|s\|^2_{g^*}\leq C_2 |s_n|^2 \|w\|^2(-\log\|w\|)^{n+1}
\]
for some constant $C_1,C_2>0$ when $\|w\|$ is sufficiently close to $0$.
Similarly, the norm of the symmetric product $\odot^k s$ has the asymptotic behaviour 
\[
C_1' |s_n|^{2k} \|w\|^{2k} (-\log\|w\|)^{k(n+1)}\leq \|\odot^k s\|^2_{\odot^k g^*}=\|s\|^{2k} _{g^*}\leq C_2' |s_n|^{2k} \|w\|^{2k} (-\log\|w\|)^{k(n+1)},
\]
for some constant $C_1',C_2'>0$ when $\|w\|$ is sufficiently close to $0$.

\subsection{Uniform Estimates of Curvature}

Consider the holomorphic tangent bundle $T_{\mathbb{B}^n}\rightarrow \mathbb{B}^n$.
Since the Bergman metric $g=(g_{i\overline{j}})$ on $\mathbb{B}^n$ is K\"ahler, the Hermitian connection $D$ with respect to $g$ agrees with the  connection. The curvature of $D$ is the tensor
\[
\Theta(T_{\mathbb{B}^n},g):=\sqrt{-1}\Theta_{\alpha \overline{\beta} i \overline{j}}e^{\alpha}\otimes \overline{e}_\beta dw^i\wedge d\overline{w}^j
\]
for any local coordinates $(w^i)$ and local holomorphic basis $(e_{\beta})$ of $T_{\mathbb{B}^n}$ and dual basis $(e^{\alpha})$ of $T^*_{\mathbb{B}^n}$. Explicitly,
\[
\Theta_{\alpha \overline{\beta} i \overline{j}}
=-\partial_i \partial_{\overline{j}} g_{\alpha \overline{\beta}}+g^{\mu \overline{\nu}} \partial_i g_{\alpha \overline{\nu}} \partial_{\overline{j}}g_{\mu \overline{\beta}},
\]
where $(g^{\gamma \overline{\delta}})$ is the conjugate inverse of $g=(g_{\alpha \overline{\beta}})$. 
Note that
$\Theta$ is independent of the choice of local holomorphic coordinates $(w^i)$.
The Euclidean coordinates at $0\in \mathbb{B}^n$ are actually complex geodesic coordinates for the Bergman metric on $\mathbb{B}^n$. If we normalize the K\"ahler form of the Bergman metric of $\mathbb{B}^n$, which is K\"ahler-Einstein, so that it has constant holomorphic sectional curvature $-2$, then it follows from homogeneity and the calculation at the point $0\in \mathbb{B}^n$ that
\begin{lemma}[cf. \cite{Mok1989}]
For the normalized Bergman metric on $\mathbb{B}^n$ with constant holomorphic sectional curvature $-2$, we have:
\[
\begin{array}{c}
-2\leq \Theta_{\alpha\overline{\alpha}\beta\overline{\beta}}\leq-1,  \\
\Ric_{i\overline{j}}=-(n+1)\delta_{ij}.
\end{array}
\]
\end{lemma}
Denote by $\mathbb{P}T_X$ the projectivized tangent bundle of $X$ obtained by fiberwise projectivization of the holomorphic tangent bundle $(T_X,g_X)\rightarrow X$
 \footnote{Note that some author may denote $\mathbb{P}T_X$ by the projectivization of the dual bundle $T^*_X$ instead of $T_X$.}. 
 One may associate the tangent bundle $(T_X,g_X)$ to the tautological line bundle $(\mathcal{O}_{\mathbb{P}T_X}(-1), \hat{g})\rightarrow \mathbb{P}T_X$ by fibrewise Hopf blow-up process.
Each point $(x,[u])\in \mathbb{P}T_X$ is described by the $x$-coordinates corresponding to the direction of base manifold and $[u]$-coordinates corresponding to the direction of the fibre $\cong \mathbb{P}^{n-1}$. The metric $\hat{g}$ on each fibre is negative definite and is actually induced from the Fubini-Study metric on $\mathbb{P}^{n-1}$. We will normalize the Fubini-Study metric to be of  Ricci constant $n$. Equip $\mathbb{P}T_X$ with the K\"ahler form
\[
\omega_{\mathbb{P}T_X}=-c_1(\mathcal{O}_{\mathbb{P}T_X}(-1),\hat{g})
=\sqrt{-1}\ddbar \log g_0.
\]
where $c_1$ denotes the first Chern class. 
It follows that $\omega_{\mathbb{P}T_X}$ is positive when restricted respectively to the fibre direction and also the base direction. It is easily seen that
\begin{lemma}
The Ricci curvature of $(\mathbb{P}T_X, \omega_{\mathbb{P}T_X})$ satisfies
\begin{equation}
-(n+1)\leq \Ric_{i\overline{j}} \leq n.
\label{RicBound}
\end{equation}
\end{lemma}
Consider the point $[u]\in \mathbb{P}^{n-1}$ in the standard chart $U_n=\{ [v^1,\dots,v^n]\in \mathbb{P}^{n-1}\mid v^n\neq 0\}$. In a neighbourhood of $(\xi,[u])$ in $\mathbb{P}T_X$, we may use the local coordinates $(\xi,[u])=(\xi^1,\dots,\xi^n,u^1,\dotsm u^{n-1})$ where $u^i=v^i/v^n$. For the line bundle
$\mathcal{O}_{\mathbb{P}T_X}(-1)\rightarrow \mathbb{P}T_X$, the holomorphic fibre coordinate is given by $\lambda=v^r$ in a neighbourhood of $[u]$.

If $e$ is a holomorphic basis of $\mathcal{O}_{\mathbb{P}T_X}(-1)$ and write $\hat{g}=g_0e^*\otimes e$, then we have
\[
g_0=\|e\|^2=g_{\alpha\overline\beta}u^\alpha\overline{u}^\beta= \sum_{\alpha,\beta=1}^{n-1} g_{\alpha \overline{\beta}} (\xi)u^\alpha\overline{u}^\beta+2Re \sum_{\alpha=1}^{n-1}g_{\alpha\overline{n}}(\xi)u^\alpha+g_{n\overline{n}}(\xi),
\]
since $u^n=1$.
The curvature of $\mathcal{O}_{\mathbb{P}T_X}(-1)$ is given by
\begin{equation}
 \Theta(\mathcal{O}_{\mathbb{P}T_X}(-1),\hat{g}) =c_1(\mathcal{O}_{\mathbb{P}T_X}(-1),\hat{g})
=-\sqrt{-1}\partial\barD \log g_0.
\end{equation}
Since $X$ and $\mathbb{P}T_X$ are homogeneous, a uniform estimate of the curvature can be obtained by the computation at one point.

Assume the fibre coordinate $\lambda$ above to be a special holomorphic fibre coordinate for $\mathcal{O}_{\mathbb{P}T_X}(-1)$ adapted to $\hat{g}$ at $(x,[u])$, which has coordinates $(\xi^1,\dots,\xi^n,0,\dots,0)$ , it is obtained by direct computation (cf. \cite{Mok1989}) that for any $\ell\in \mathbb{N}$,
\begin{equation}
\begin{array}{rcl}
\Theta(\mathcal{O}_{\mathbb{P}T_X}(\ell),(\hat{g}^*)^{\ell})(x,[u])
&=&\ell\sqrt{-1}\displaystyle \left(\sum_{\alpha=1}^{n-1} du^\alpha\wedge d\overline{u}^\alpha - \sum_{i,j=1}^{n} \Theta_{n\overline{n}i\overline{j}} d\xi^i \wedge d\overline{\xi}^j\right).
\end{array}
\label{curvsimp}
\end{equation}
It also implies that
the hyperplane section line bundle $(\mathcal{O}_{\mathbb{P}T_X}(1),\hat{g}^*)\rightarrow \mathbb{P}T_X$ is of positive curvature.

\subsubsection{Asymptotic Behaviour of Curvature}
To obtain our Main Theorem (Theorem \ref{main}), however, it is not enough to use the uniform estimates coming from the calculation under special coordinates at one point and homogeneity. We will need to consider the asymptotic behaviour of the canonical metric induced by the normalized Bergman metric. Thus it is necessary to use the full calculation of the curvature. First note that we may write
\begin{equation}
g=(g_{\alpha\overline{\beta}})=\frac{1}{(-\log\|w\|)^2}(h_{\alpha\overline{\beta}}),
\label{gasym}
\end{equation}
where
\[
(h_{\alpha\overline{\beta}})=\left(
\begin{array}{cc}
(\frac{2\pi}{\tau}(-\log\|w\|)\delta_{\alpha\beta}-(\frac{2\pi}{\tau})^2w_\alpha\overline{w}_\beta)  & \frac{2\pi}{\tau}\frac{\overline{w}_\beta}{\overline{w}_n}\\
\frac{2\pi}{\tau}\frac{w_\alpha}{w_n} & \frac{1}{|w_n|^2}
\end{array}
\right), \quad 1\leq \alpha,\beta \leq n-1.
\]
When $\|w\|\rightarrow 1$, each component $g_{\alpha\overline{\beta}}$ is of growth $1/ \delta^2$ where 
\[
\delta(w):=1-\|w\|
\] 
is the distance from to the boundary $\{\|w\|=1\}$. Then
\begin{equation}
\begin{array}{lll}
 \Theta(\mathcal{O}_{\mathbb{P}T_X}(-1),\hat{g}) &=&c_1(\mathcal{O}_{\mathbb{P}T_X}(-1),\hat{g}) \\
&=&-\sqrt{-1}\partial\barD \log g_0 \\
&=& -\frac{\sqrt{-1}}{g_0^2}[\left(g_0 \pd_i\pd_{\overline{j}} g_0-\pd_i g_0\pd_{\overline{j}}g_0\right) d\xi^i\wedge d\overline{\xi}^j
+  \left(g_0 \pd_i\pd_{\overline{\beta}} g_0-\pd_i g_0\pd_{\overline{\beta}}g_0\right) d\xi^i\wedge d\overline{u}^\beta\\
& &+  \left(g_0 \pd_i\pd_{\overline{\alpha}} g_0-\pd_\alpha g_0\pd_{\overline{j}}g_0\right) du^\alpha\wedge d\overline{\xi}^j
+  \left(\pd_\alpha\pd_{\overline{\beta}} g_0-\pd_\alpha g_0\pd_{\overline{\beta}}g_0\right) du^\alpha\wedge d\overline{u}^\beta],
\end{array}
\label{curv}
\end{equation}
where $1\leq i,j,\alpha,\beta \leq n$. The numerator of the R.H.S. of Equation (\ref{curv}) is $-\sqrt{-1}$ times
\begin{equation}
\begin{array}{lll}
& &g_0\partial\barD g_0-\partial g_0\wedge \barD g_0 \\
 &=& u^\alpha\overline{u}^\beta u^{\gamma}\overline{u}^\delta \left( g_{\gamma \overline{\delta}}\pd_i\pd_{\overline{j}}g_{\alpha\overline{\beta}}-\pd_i g_{\alpha \overline{\beta}}\pd_{\overline{j}} g_{\gamma \overline{\delta}} \right) d\xi^i\wedge d\overline{\xi}^j 
 + u^\alpha u^{\gamma}\overline{u}^\delta \left( \pd_{\overline{j}} g_{\alpha\overline{\beta}}-\pd_{\overline{j}} g_{\gamma \overline{\delta}}\right) du^\alpha\wedge d\overline{\xi}^j \\
 &&+ \overline{u}^\beta u^{\gamma}\overline{u}^\delta \left(\pd_ig_{\alpha\overline{\beta}}-\pd_i g_{\alpha \overline{\delta}}\right) d\xi^i\wedge d\overline{u}^\beta
 + u^{\gamma}\overline{u}^\delta \left(g_{\alpha \overline{\beta}}g_{\gamma \overline{\delta}}-g_{\alpha\overline{\delta}}g_{\gamma \overline{\beta}}\right) du^\alpha\wedge d\overline{u}^\beta,
\end{array}
\label{numer}
\end{equation}
and the denominator is
\begin{equation}
g_0^2=g_{\alpha\overline{\beta}}g_{\gamma\overline{\delta}}u^\alpha \overline{u}^\beta u^\gamma\overline{u}^\delta,
\label{denom}
\end{equation}
where $1\leq \alpha,\beta,\gamma,\delta,i,j\leq n$. Since both the numerator and the denominator of the R.H.S. of the Equation (\ref{curv}) are polynomial of degree $4$ in $u^i,\overline{u}^j$'s and all possible combinations of the $u^i,\overline{u}^j$'s in the monomials appear once, we know that
\begin{lemma}
Each component $\Theta(\mathcal{O}_{\mathbb{P}T_X}(-1),\hat{g})_{i\overline{j}\alpha\overline{\beta}}$ is bounded as $|(u^1,\dots,u^{n-1},1)|\rightarrow \infty$ in any $u$-directions (i.e. the fibre directions).
\label{fibreest}
\end{lemma}
\begin{remark}
For the purpose of studying asymptotic behaviour of the curvatures in some inequalities, it suffices to focus on the base directions.
\end{remark}

We will make use of two kinds of asymptotic behaviours of $\Theta(\mathcal{O}_{\mathbb{P}T_X}(-1),\hat{g})$. Let 
\[
\Omega:=\coprod_b \Omega_b
\] 
be a disjoint union of the canonical neighbourhoods of each component $T_b$ of the infinity , where $T_b$ corresponds to the cusps $b$ of $\Gamma$. Thus each $\Omega_b$ has a disk bundle structure over $T_b$ with radius less or equal than the canonical radius of $b$. We are going to view $\Omega-\coprod_b T_b \subset \mathbb{P}T_X$. We will also let
\[
\Omega':= \mathbb{P}T_X-\Omega
\]
to be the interior relatively compact part.

\subsubsection*{Curvature in the Interior}
Suppose $p\in \Omega'$. Then $p$ lifts to a point in a fundamental domain $\mathcal{F}\subset \mathbb{B}^n$ of $\Gamma$. We may assume the origin $0\in \mathcal{F}$. These is a metric ball neighbourhood of $p$ in $X$ of maximum radius measured with respect to the canonical metric $g_X$. The radius of the ball neighbourhood depends on the injectivity radius at the point $p$ with respect to $g_X$. 
Thus there exists $\delta>0$, which depends on the injectivity radius at $p$ with respect to $g_X$, such that $p\in B(0,1-\delta)\subset\mathcal{F}\subset \mathbb{B}^n$, where $B(z,r)$ denotes the Euclidean ball centred at $z$ with radius $r$. Now we have
\[
-\log\|w\|=\frac{2\pi}{\tau}(\im(\xi_n)-|\xi'|^2)=\frac{2\pi}{\tau}\frac{1-|z|^2}{|1-z_n|^2}=\frac{2\pi}{\tau}\frac{(1+|z|)(1-|z|)}{|1-z_n|^2}=:f(z)\delta(z),
\]
where $\delta(z):=1-|z|$ and $f(z)$ is non-vanishing. In view of Equation (\ref{gasym}), it then follows that
\[
\partial_i g_{\alpha\overline{\beta}}=\frac{1}{\delta(z)^3}\psi_{i,\alpha\overline{\beta}}, \quad 
\partial_{\overline{j}}g_{\alpha\overline{\beta}}=\frac{1}{\delta(z)^3}\psi_{\overline{j},\alpha\overline{\beta}}, \quad
\partial_i\partial_{\overline{j}}g_{\alpha\overline{\beta}}=\frac{1}{\delta(z)^4}\phi_{i\overline{j},\alpha\overline{\beta}},
\]
where $\psi_{i,\alpha\overline{\beta}},\psi_{\overline{j},\alpha\overline{\beta}},\phi_{i\overline{j},\alpha\overline{\beta}}$ are functions that are bounded as $|z|\rightarrow 1$.
By the explicit Equation (\ref{numer}) and (\ref{denom}), it follows that the growth of each component of the curvature is
\[
\Theta(\mathcal{O}_{\mathbb{P}T_X}(-1),\hat{g})_{i\overline{j}\alpha\overline{\beta}}\sim 1/\delta^2, \quad \text{when $|z|\rightarrow 1$},
\]
where $\delta(z)=1-|z|$.
\subsubsection*{Curvature Near Infinity}
Suppose $p\in \Omega_b\subset \Omega$ is a point in some fundamental cusp neighbourhood $\Omega_b=\{\|w\|\leq r_b\}$ (where $0<r_b<1$) of a component $T_b$ of the boundary divisor of the Mumford compactification $\overline{X}$. 
By similar calculations as in the above, we have the asymptotic behaviour 
\[
\Theta(\mathcal{O}_{\mathbb{P}T_X}(-1),\hat{g})_{i\overline{j}\alpha\overline{\beta}}\sim 1/\delta^2, \quad \text{when $\|w\|\rightarrow 1$},
\]
where $\delta(w)=1-\|w\|$. Geometrically, this situation corresponds to the approaching from cups neighbourhood to the interior. It is similar to the one dimensional analogue, i.e., the Poincar\'e metric for punctured unit disk. The analogous situation is the approaching of an interior point to the  boundary $\{|z|=1\}$.

We record the asymptotic behaviours of the curvature as follows:
\begin{lemma}
Let $\hat{g}$ be the Hermitian metric on $\mathbb{P}T_{X}$ induced from the K\"ahler metric on $\mathbb{P}T_X$ (which is induced from the canonical K\"ahler-Einstein metric on the base $X=\mathbb{B}^n/\Gamma$ and the Fubini-Study metric on each fibre $\mathbb{P}^{n-1}$). Let $\Omega:=\coprod_b \Omega_b$ and $\Omega':= \mathbb{P}T_X-\Omega$.
\begin{enumerate}
\item Suppose $p\in \Omega'$ is lifted to a point $z\in \mathbb{B}^n$. Then 
\begin{equation}
\Theta(\mathcal{O}_{\mathbb{P}T_X}(-1),\hat{g})_{i\overline{j}\alpha\overline{\beta}}\sim 1/\delta^2, \quad \text{when $|z|\rightarrow 1$},
\label{intasymp}
\end{equation}
where $\delta(z)=1-|z|$.
\item Suppose $w\in \Omega$ is contained in a fundamental neighbourhood $\Omega_b\subset \{0\leq \|w\|\leq 1\}$. Then
\begin{equation}
\Theta(\mathcal{O}_{\mathbb{P}T_X}(-1),\hat{g})_{i\overline{j}\alpha\overline{\beta}}\sim 1/\delta^2, \quad \text{when $\|w\|\rightarrow 1$},
\label{infasymp}
\end{equation}
where $\delta(w)=1-\|w\|$.
\end{enumerate}
\end{lemma}

\section{The Construction of Symmetric Differentials Vanishing at Infinity}\label{constr}
In this section, we will construct symmetric differentials on $\overline{X}$ which vanishes at infinity $D=\overline{X}-X$. The construction is done by methods of $L^2$-estimates, which relies on the existence theorem:
\begin{theorem}[Andreotti-Vesentini \cite{AV1965} and H\"ormander \cite{Hor1965}]
Let $(M,g)$ be a complete K\"ahler manifold with K\"ahler form $\omega$ and Ricci form $\Ric(\omega)$. Let $(\mathcal{L},h)\rightarrow M$ be a Hermitian holomorphic line bundle with curvature form $\Theta(\mathcal{L},h)$. 
Suppose $\varphi$ is a smooth function on $M$ and $c$ is a positive continuous function on $X$ such that
\begin{equation}
\Theta(\mathcal{L},h)+\Ric(\omega)+\sqrt{-1}\partial\barD \varphi \geq c\omega 
\label{basicl2}
\end{equation}
everywhere on $M$. Let $f$ be a $\barD$-closed square integrable $\mathcal{L}$-valued $(0,1)$-form on $M$ such that $\int_M \frac{\| f\|^2 _g}{c} dV_g<\infty$.
Then there exists a section $u$ of $\mathcal{L}\rightarrow M$ solving the inhomogeneous problem $\barD u=f$ and also satisfying the estimate
\begin{equation}
\int_M \|u\|_g^2e^{-\varphi} dV_g\leq \int_M \frac{\| f\|^2 _g}{c} e^{-\varphi} dV_g<\infty.
\end{equation}
Furthermore, $u$ is smooth whenever $f$ is.
\label{hor}
\end{theorem}

\begin{remark}
When the weight function $\varphi$ is singular, the above existence theorem also holds provided that $\varphi$ can be approximated by a monotonically decreasing sequence $\{\varphi_\epsilon\}$ where each of $\varphi_\epsilon$ is smooth and plurisubharmonic. We will consider weight function of the form $\varphi\sim \log |z|^2$. The approximation can then be taken as $\varphi_\epsilon\sim\log (|z|^2+\epsilon)$ for $\epsilon\rightarrow 0$.
\end{remark}

We will carry on the construction by first assuming the basic estimates in Theorem \ref{hor}, leaving its justification to next section.

\subsection{The Construction Assuming the Basic Estimates}
Let $U\subset X \subset \mathbb{P}T_X$ be an open neighbourhood of the boundary divisor $T_b=\{w_n=0\}$ corresponding to a cusp $b$ of $\Gamma$. Let $\ell, n, \gamma$ to be respectively the degree of symmetric power of the cotangent bundle, the dimension of $X=\mathbb{B}^n/\Gamma$ and the degree of singularity in the weight function
\begin{equation}
\varphi:=\gamma \chi \log \|w\|^2,
\label{weight}
\end{equation}
where $\chi$ is a cutoff function near infinity to be specified explicitly later and $\gamma$ is a nonnegative integer.

Assume $\varepsilon>0$ such that on $\mathbb{P}T_X$, the basic estimates
\[
\Theta(\mathcal{O}_{\mathbb{P}T_{X}}(\ell),(\hat{g}^*)^{\ell})+\Ric(\omega_{\mathbb{P}T_X})+\sqrt{-1}\ddbar\varphi \geq \varepsilon \omega_{\mathbb{P}T_X}
\]
holds.

\begin{remark}
By our technique, we need to require $\ell\geq n+1+\varepsilon$.
\end{remark}
Suppose 
\[
\begin{array}{lc}
(\dagger): &\quad \exists \nu\in H^0(U-T_b,\mathcal{O}(\ell)|_{U-{T_b}})
\end{array}
\]
a nowhere zero holomorphic section so that it extends to and vanishes at $T_b$ of order exactly $m>0$.

Take a smooth cut-off function $\chi: \mathbb{P}T_X\rightarrow [0,1]$ such that for some neighbourhood $U_1$ of $(\mathbb{P}T_X-U)\cap X$ 
\[
\chi(p)=\begin{cases} 1, \text{if $p\in U$}\\ 0, \text{if $p\in U_1$.} \end{cases}
\]
We get a global smooth section $\eta:=\chi \nu\in \mathcal{C}^{\infty}_c(\mathbb{P}T_X,  \mathcal{O}_{\mathbb{P}T_X}(\ell))$.
Since $\barD{\eta}=\nu \barD{\chi}+\chi\barD{\nu}$. On $U$, $\chi\equiv 1$ and $\eta|_U=\nu \in \Gamma(U, \mathcal{O}_{\mathbb{P}T_X}(\ell)|_U)$, it follows that $\barD{\eta}\equiv 0$ . On $U_1$, $\chi\equiv 0$, which also implies that $\barD{\eta}\equiv 0$.
Thus $\barD{\eta}$ has compact support $\Supp(\barD{\eta})\Subset \mathbb{P}T_X-(U\cup U_1)$.
It follows that
\[
\int_{\mathbb{P}T_X} \|\barD{\eta}\|^2e^{-\varphi}\omega^{2n-1}_{\mathbb{P}T_X}<\infty.
\]
Now consider the following inhomogeneous partial differential equation of unknown $u$ on $\mathbb{P}T_X$:
\begin{equation}
\barD{u}=\barD\eta.
\label{pde}
\end{equation}
Hence by Theorem \ref{hor}, a solution $u$ to Equation $(\ref{pde})$ exists and it satisfies the estimates
\begin{equation}
\int_{\mathbb{P}T_X} \|u\|^2e^{-\varphi}\omega^{2n-1}_{\mathbb{P}T_X}
\leq \frac{1}{\varepsilon}\int_{\mathbb{P}T_X} \|\barD{\eta}\|^2e^{-\varphi}\omega^{2n-1}_{\mathbb{P}T_X}<\infty.
\label{estsect}
\end{equation}
Recall that the metric $(\hat{g}^*)^\ell$ is induced from the dual metric $\odot^\ell g^*$ on $S^\ell (T^*_X)$ and we have an isomorphism of sections $H^0(\mathbb{P}T_X, \mathcal{O}_{\mathbb{P}T_{X}}(\ell))\cong H^0(X,S^\ell (T^*_X))$. Thus we may identify $u$ to a section $u_0\in H^0(X,S^\ell (T^*_X))$ and $\eta$ to $\eta_0^\ell\in H^0(X,S^\ell (T^*_X))$.
Then Inequality (\ref{estsect}) implies that
\begin{equation}
\int_X \|u_0\|_{\odot^\ell g^*}^2 e^{-\varphi}dV_{g_X}<\infty.
\label{estsect0}
\end{equation}
Since the asymptotic behaviour of the metric $g_X$ when approaching to $T_b=\{w_n=0\}$ is donminanted by the $w_n$-coordinate, we may write without loss of generality that 
\[
u_0=f(dw_n)^{\odot \ell}, \quad \text{in a small neighbourhood $\Delta^n(\varepsilon)\cong \Omega_0\subset \Omega_b^{(N)}$ of $T_b$.}
\]
Now
\begin{eqnarray*}
\infty &>&\int_{\Omega_0} \|u_0\|_{\odot^\ell g^*}^2 e^{-\varphi}dV_{g_X} \\
&=&\int_{\Omega_0} \|f(dw_n)^{\odot \ell}\|_{\odot^\ell g^*}^2 e^{-\varphi}dV_{g_X} \\
&\geq &  const\cdot \int_{\Delta^n(\varepsilon)}|f|^{2} \frac{1}{\|w\|^{2(-\ell+\gamma+1)}}(-\log \|w\|)^{(\ell-1)(n+1)} dV. \\
\end{eqnarray*}
In order to avoid $u_0$ to be the trivial solution, we want $f$ to vanish at $\{w_n=0\}$ to order at least $m+1$ for some $m\geq 0$. This is the same as requiring
\[
-\ell+\gamma+1\geq m+1 \Leftrightarrow \gamma\geq \ell+m.
\]
Let
\[
\sigma:=\eta-u. 
\]
Then $\barD{\sigma}=\barD{\eta}- \barD u\equiv 0$ on $\mathbb{P}T_X$, which implies that $\sigma\in H^0(\mathbb{P}T_X,\mathcal{O}_{\mathbb{P}T_X}(\ell))$ such that it vanishes at $T_b=\{w_n=0\}$ to order $m$.

\subsection{The Justification of $(\dagger)$}
\begin{proposition}
There exists an open neighbourhood $U\subset \overline{X} \subset \mathbb{P}T_{\overline{X}}$ of the boundary divisor $T_b=\{w_n=0\}$ corresponding to a cusp $b$ of $\Gamma$ such that there is a nowhere zero holomorphic section
\[
\nu\in H^0(U-T_b,\mathcal{O}_{\mathbb{P}T_{\overline{X}}}(\ell)|_{U-{T_b}})
\]
which extends and vanishes to order exactly $\ell m>0$ at $T_b$.
\label{locdiff}
\end{proposition}
\begin{proof}
We are going to construct a local holomorphic section 
\[
\nu\in H^0(U-T_b,\mathcal{O}(\ell)|_{U-{T_b}})\cong H^0(U-T_b,S^\ell T^*_X|_{U-{T_b}}),
\] 
which is nowhere zero in  $U-T_b$, extends to $T_b$ and vanishes precisely to order $m>0$ at the infinity $T_b$.

Take $T_b\subset U=\Omega_b\subset \overline{X}$, where $\Omega_b$ is a level set of the norm 
\[
\|(w',w_n)\|:= e^{\frac{2\pi}{\tau}|w'|^2}|w_n|.
\]
The normal bundle 
\[
\mathcal{N}_{T_b|\overline{X}}\rightarrow T_b
\] 
is negative under the Hermitian metric $\mu$. By Lefschetz's embedding theorem for abelian variety (cf. for example \cite{BL2004} Theorem 4.5.1), $\mathcal{N}_{T_b|\overline{X}}^{-m}$ is very ample for every $m\geq 3$. In fact $\mathcal{N}_{T_b|\overline{X}}^{-2}$ is already base-point free.

Let $L=\mathcal{N}_{T_b|\overline{X}}$.
Identify $T_b\subset L$ as the zero section.
Write $T_b=\bigcup\limits_\alpha U_\alpha$ as a union of open covers. We have the local trivialization
\[
L|_{U_\alpha}=U_\alpha\times \mathbb{C}
\]
Let $(z,w_\alpha)\in U_\alpha\times \mathbb{C}$ be the local coordinates and $e_\alpha$ be a holomorphic basis of $L|_{U_\alpha}$. Then the coordinate functions $\{w_\alpha\}$ may be viewed as the sections of $L^{-1}$ and they satisfy
\[
w_\alpha e_\alpha = w_\beta e_\beta, \quad \text{on $L|_{U_\alpha\cap U_\beta}$}.
\]
Fix $m\geq 1$. Since $L^{-(m+1)}\rightarrow T_b$ is base-point free, there is a holomorphic section $s\in H^0(T_b,L^{-(m+1)})$ such that $s(p)\neq 0$ at a particular point $p\in T_b$. Hence $s$ is nontrivial and nonvanishing on $T_b$. Write
\[
s=s_\alpha e_\alpha^{-(m+1)}.
\]
Over $U_\alpha\cap U_\beta$, we have
\[
s_\alpha e_\alpha^{-(m+1)}=s_\beta e_\beta^{-(m+1)}.
\]
On $L|_{U_\alpha}\cong U_\alpha\times \mathbb{C}$, define
\[
\begin{array}{cccc}
f_\alpha: &U_\alpha\times \mathbb{C}&\rightarrow &\mathbb{C} \\
&(z,w_\alpha) &\rightarrow  &s_\alpha(z)w_\alpha^{m+1}.
\end{array}
\]
Then $f=\{f_\alpha\}$ is a well-defined holomorphic function on $L$ since on $(U_\alpha\cap U_\beta)\times \mathbb{C}$,
\[
s_\alpha(z)w_\alpha^{m+1}=s_\beta e_\beta^{-(m+1)} e_\alpha^{m+1}w_\beta^{m+1} e_\beta^{m+1}e_\alpha^{-(m+1)}=s_\beta(z)w_\beta^{m+1}.
\]
View $\Omega_b\subset L=\mathcal{N}_{T_b|\overline{X}}$ as an open subset.
By construction, $f$ is nonconstant on $\Omega_b-T_b$ and vanishes precisely to order $m+1$ on $T_b$. 
On $L|_{U_\alpha}\cong U_\alpha\times \mathbb{C}$,
\begin{eqnarray*}
df(z,w_\alpha)
&=& w_\alpha^{m} [w_\alpha ds_\alpha+(m+1)s_\alpha dw_\alpha ] \\
(df)^{\odot \ell}
&=& w_\alpha^{\ell m} \left[\sum\limits_{k=0}^{\ell} C^{\ell}_k (w_\alpha^k)((m+1)s_\alpha)^{\ell-k}  (ds_\alpha)^{\odot k}\odot(dw_\alpha)^{\odot (\ell-k)}\right]
\end{eqnarray*}
So the differential $df\in H^0(\Omega_b,T^*\Omega_b)$ vanishes precisely to order $m$ on $T_b$ and it follows that
\[
(df)^{\odot \ell}\in H^0(\Omega_b,S^\ell T^*\Omega_b)
\]
vanishes precisely to order $\ell m$ on $T_b$.
\end{proof}

\section{Ampleness Modulo Infinity}
\label{amp}
In this section, we will prove another technical result which is essentially the justification of the basic estimates for the existence Theorem \ref{hor}
under conditions about canonical radii of cusps and injectivity radius of interior of $\overline{X}$. 

We will use the weight function $\varphi$ defined by Equation $(\ref{weight})$, thus the term $\sqrt{-1}\ddbar\varphi$ is negative only in the base direction. By Lemma \ref{RicBound}, the Ricci curvature $\Ric(\omega_{\mathbb{P}T_X})$ is bounded. In particular, $\Ric(\omega_{\mathbb{P}T_X})$ is positive in the fibre directions since each fibre is equipped with the Fubini-Study metric and negative in the base directions, which is the canonical metric induced by the normalized Bergman metric. It follows that the terms
\[
\Ric(\omega_{\mathbb{P}T_X})+\sqrt{-1}\ddbar\varphi
\]
is the most negative along the base directions. Thus we only need to justify $(*)$ along the base directions.

\begin{proposition}
Let $X=\mathbb{B}^n/\Gamma$, $b\in \partial \mathbb{B}^n$ be a cusp of $\Gamma$, $\tau>0$ be the generator of the one parameter subgroup $\Gamma\cap [U_b,U_b]$, where $U_b=\rad P_b$ is the unipotent radical of the parabolic subgroup $P_b\subset \Aut(\mathbb{B}^n)$ ( $P_b$ is the isotropy subgroup of $b$ in $\Aut(\mathbb{B}^n)$).
Suppose $\ell,\gamma$ are nonnegative integers and $\varepsilon>0$ such that 
\[
\ell-n-1-\varepsilon> 0.
\]
Then there is a constant $r^*=r^*(n)>0$ such that if $\|w\|\geq r^*$ for any point $(w,[u])\in \Supp(\chi)\subset X\subset \overline{X}\subset \mathbb{P}T_{\overline{X}}$, then
\[
(*):\quad \Theta(\mathcal{O}_{\mathbb{P}T_{X}}(\ell),(\hat{g}^*)^{\ell})+\Ric(\omega_{\mathbb{P}T_X})+\sqrt{-1}\ddbar\varphi\geq \varepsilon \omega_{\mathbb{P}T_X}
\label{basic}
\]
everywhere on $\mathbb{P}T_X$.
\label{keyprop1}
\end{proposition}

\begin{proof}
Let $\varepsilon >0$. 
By Equation (\ref{curvsimp}), Lemma \ref{RicBound} and the assumption $\ell-n-1-\varepsilon> 0$, we have
\[
\begin{array}{rcl}
(*_1): &&\Theta(\mathcal{O}_{\mathbb{P}T_{X}}(\ell),(\hat{g}^*)^{\ell})+\Ric(\omega_{\mathbb{P}T_X})-\varepsilon \omega_{\mathbb{P}T_X} \\
&\geq & (\ell-n-1-\varepsilon) \sqrt{-1}\ddbar\log g_0>0.
\end{array} 
\]
Thus it remains to show that the curvature term $\sqrt{-1}\ddbar\log g_0$ dominates the weight function term $\sqrt{-1}\ddbar\varphi$.

View $\overline{X}\subset \mathbb{P}T_{\overline{X}}$ by identifying $\overline{X}$ with the zero section in $\mathbb{P}T_{\overline{X}}\rightarrow \overline{X}$.
Recall that for each $u\geq u_b$ (the canonical height of $b$), we may view $\{w_n=0\}=T_b\subset \Omega^{(u)}\subset \overline{X} \subset \mathbb{P}T_{\overline{X}}$, where
\[
\Omega^{(u)}=\{ w=(w',w_n)\in \mathbb{C}^n \mid \|w\|=e^{\frac{2\pi}{\tau}|w'|^2}\cdot |w_n|<e^{\frac{-2\pi}{\tau}u}\}/\varphi(\Gamma\cap U_b)
\]
is a level set of $\overline{\mu}$ (recall that $\overline{\mu}$ is the Hermitian metric of the line bundle over the boundary abelian variety $T_b$ which induces the norm $\|\cdot \|$).
For $u_1<u_0$, we have $\Omega^{(u_1)}\supset  \Omega^{(u_0)}\supset T_b$.
Define on $\Omega^{u}\subset \overline{X}\subset \mathbb{P}T_{\overline{X}}$ ($u<u_1<u_0)$ the cutoff function 
\[
\chi(\|w\|)=
\begin{cases} 
1, \quad \|w\|\leq e^{\frac{2\pi}{\tau}u_0}:=r_0, \\ 
0, \quad \|w\|\geq e^{\frac{2\pi}{\tau}u_1}:=r_1.
\end{cases}
\]
Note that
\[
\begin{array}{rcl}
\Supp(\chi)\subset \Omega^{(u_1)}-\Omega^{(u_0)}&=&\{ w\in \mathbb{C}^n \mid r_0 \leq \| w \| < r_1\}/\varphi(\Gamma\cap U_b). \\
\end{array}
\]
By construction, $\chi$ depends only on $\|w\|$. Moreover, $\frac{\partial \chi}{\partial \|w\|}<0$ on $\Omega^{(u_1)}-\Omega^{(u_0)}$.
Define
\[
\delta(w):= 1-\|w\|,
\]
which is just the distance from $w\in \Omega^{(u)}$ to the boundary $\{\|w\|=1\}$.
Let $r_1=1-\delta$ and $r_0=1-2\delta$. From standard arguments of mollifiers (cf. \cite{Hor1990} Theorem 1.4.1 and the note after the proof), we have the uniform bounds
\[
|\frac{\partial \chi}{\partial w_i}|, |\frac{\partial \chi}{\partial \overline{w}_j}| \leq \frac{1+\epsilon}{\delta}< \frac{2}{\delta}, \quad 
|\frac{\partial^2 \chi}{\partial w_i\partial \overline{w}_j}|< \frac{1+\epsilon}{\delta^2}< \frac{2}{\delta^2}, \quad \forall 1\leq i,j \leq n
\]
for any $0<\epsilon<1$.
For some $\gamma\in \mathbb{N}$, let
\[
\varphi:= \gamma \chi \log \|w\|^2,
\]
where $\|w\|=\|(w',w_n)\|=e^{\frac{2\pi}{\tau}|w'|^2}\cdot |w_n|$ is the norm induced by the Hermitian metric $\overline{\mu}$ on $\Omega^{(u)}\rightarrow T_b$.

We need to obtain a pointwise estimate of $\sqrt{-1}\ddbar \varphi$ in $\Omega^{(u_1)}-\Omega^{(u_0)}$.
First note that by chain rule,
\[
\begin{array}{rcl}
\ddbar \chi(\|w\|) &=& \chi''(\|w\|)\partial \|w\|\wedge \barD \|w\|+\chi'(\|w\|)\ddbar \|w\|
\end{array}
\]
The Hessian $\sqrt{-1}\ddbar\|w\|$ can be expressed in terms of gradient square $\sqrt{-1}\partial \|w\|\wedge \barD\|w\|$ as follows
\[
\ddbar\|w\|=\frac{2\pi}{\tau}\|w\|\ddbar|w'|^2+\frac{\partial \|w\|\wedge \barD\|w\|}{\|w\|}.
\]
Note that $\Supp(\varphi)\subset \Omega^{(u_1)}-\Omega^{(u_0)}$.
So $\varphi$ is a smooth function on $\mathbb{P}T_{\overline{X}}$ such that
\begin{eqnarray*}
(*_2):
&&\sqrt{-1}\ddbar \varphi\\
&=& \displaystyle \gamma\sqrt{-1}
\left(\chi \ddbar \log \|w\|^2+\log \|w\|^2 \ddbar \chi+\partial \log \|w\|^2\wedge \barD \chi  +\partial \chi\wedge \barD \log \|w\|^2 \right) \\
&=& \displaystyle\left[\chi \frac{4\pi}{\tau}+\frac{4\pi}{\tau}\chi'(\|w\|)\|w\|\log \|w\|\right]\gamma\sqrt{-1}\ddbar |w'|^2 \\
&& + \left[2\chi''(\|w\|)\log \|w\|+\frac{4\chi'(\|w\|)}{\|w\|}+\frac{2\chi'(\|w\|)\log\|w\|}{\|w\|}\right]\gamma\sqrt{-1}\partial\|w\|\wedge \barD\|w\| \\
&=  & \displaystyle\left[\chi \frac{4\pi}{\tau}+\frac{4\pi}{\tau}\chi'(\|w\|)\|w\|\log \|w\|\right]\gamma\sqrt{-1}\ddbar |w'|^2 \qquad (\text{positive term})\\
&& +\left[\frac{2\chi'(\|w\|)\log\|w\|}{\|w\|}\right]\gamma\sqrt{-1}\partial\|w\|\wedge \barD\|w\|  \qquad (\text{positive term})\\
&& + \displaystyle \left[2\chi''(\|w\|)\log \|w\|+\frac{4\chi'(\|w\|)}{\|w\|}\right]\gamma\sqrt{-1}\partial\|w\|\wedge \barD\|w\| \qquad (\text{negative term}).
\end{eqnarray*}
For any $0<1-2\delta\leq \|w\|< 1-\delta<1$, as $\delta\rightarrow 0$,
\begin{eqnarray*}
&\frac{1}{1-\delta}<\frac{1}{\|w\|}\leq \frac{1}{1-2\delta},\quad -\frac{2}{\delta}<\chi'(\|w\|) <0, \quad -\frac{2}{\delta^2}<\chi''(\|w\|)<\frac{2}{\delta^2},\\
&-2\delta(1+o(\delta))=\log (1-2\delta)\leq \log\|w\| <\log (1-\delta)=-\delta(1+o(\delta))<0.
\end{eqnarray*}
Thus
\begin{eqnarray*}
&&2\chi''(\|w\|)\log \|w\|\geq
\begin{cases}
2(\frac{2}{\delta^2})(-2\delta(1+o(\delta)))=-\frac{8}{\delta}(1+o(\delta)), \quad \text{if $0<\chi''<\frac{2}{\delta^2}$},\\
0, \quad \text{if $-\frac{2}{\delta^2}< \chi''\leq 0$},
\end{cases}\\
&\Rightarrow & 2\chi''(\|w\|)\log \|w\|
\geq -\frac{8}{\delta}(1+o(\delta))=-\frac{8}{\delta}+O(1).
\end{eqnarray*}
and
\[
\frac{4\chi'(\|w\|)}{\|w\|}\geq 4(-\frac{2}{\delta})(\frac{1}{1-2\delta})=-8(\frac{1}{\delta}+\frac{2}{1-2\delta})=-\frac{8}{\delta}+O(1).
\]
Thus the negative term
\[
2\chi''(\|w\|)\log \|w\|+\frac{4\chi'(\|w\|)}{\|w\|}
\geq  -\frac{8}{\delta}+O(1)-\frac{8}{\delta}+O(1)
=-\frac{16}{\delta}+O(1).
\]
It follows that
\begin{eqnarray*}
R.H.S. (*_2) &\geq  & \displaystyle\left[\chi \frac{4\pi}{\tau}+\frac{4\pi}{\tau}\chi'(\|w\|)\|w\|\log \|w\|\right]\gamma\sqrt{-1}\ddbar |w'|^2 \qquad (\text{positive term}) \\
&& +\left[\frac{2\chi'(\|w\|)\log\|w\|}{\|w\|}\right]\gamma\sqrt{-1}\partial\|w\|\wedge \barD\|w\| \qquad (\text{positive term}) \\
&& + \displaystyle \left[-\frac{16}{\delta}+O(1)\right]\gamma\sqrt{-1}\partial\|w\|\wedge \barD\|w\| \qquad (\text{negative term}).
\label{estweight}
\end{eqnarray*}
To verify the basic estimates with weights $(*)$, it suffices to verify for points in $\Supp \chi\subset \Omega^{(u_1)}-\Omega^{(u_0)}$.
All negativity comes from the weight function which lies in base directions (i.e., the $w$-directions).
By the uniform estimate of the curvature via Equation (\ref{curvsimp}) (or the fact that the K\"ahler metric along fibre $\cong\mathbb{P}^{n-1}$ is essentially the Fubini-Study metric), we know that $(*)\geq 0$ in the fibre directions (i.e., the $u$-directions). Thus it remains to consider the base directions. 

Combining $(*_1)$ and $(*_2)$, we have
\begin{equation*}
\begin{array}{rcl}
(**): &&\Theta(\mathcal{O}_{\mathbb{P}T_{X}}(\ell),(\hat{g}^*)^{\ell})+\Ric(\omega_{\mathbb{P}T_X})-\varepsilon \omega_{\mathbb{P}T_X}+\sqrt{-1}\ddbar\varphi \\
&\geq & (\ell-n-1-\varepsilon) \sqrt{-1}\ddbar\log g_0  \\
&&+  \displaystyle\left[\chi \frac{4\pi}{\tau}+\frac{4\pi}{\tau}\chi'(\|w\|)\|w\|\log \|w\|\right]\gamma\sqrt{-1}\ddbar |w'|^2 \\
&&+ \left[\frac{2\chi'(\|w\|)\log\|w\|}{\|w\|}\right]\gamma\sqrt{-1}\partial\|w\|\wedge \barD\|w\|\\
&& + \displaystyle \left[-\frac{16}{\delta}+O(1)\right]\gamma\sqrt{-1}\partial\|w\|\wedge \barD\|w\|. \\
\end{array} 
\end{equation*} 
Observe that the negative terms on the right hand side of $(**)$ has asymptotic behaviour at worst $ 1/\delta$.
By the asymptotic behaviour (\ref{infasymp}), when $\|w\|\rightarrow 1$, the curvature $\sqrt{-1}\ddbar \log g_0$ has asymptotic behaviour $1/\delta^2$. Hence the curvature $\sqrt{-1}\ddbar \log g_0$ dominates the negativity coming from the weight functions. Thus we have verified the basic estimates $(*)$ for the points along the fibre in the chart $U_n=\{ [v^1,\dots,v^n]\in \mathbb{P}^{n-1}\mid v^n\neq 0\}$. Now by Lemma \ref{fibreest}, the above argument holds for any point $(u^1,\dots,u^{n-1},1)$, even when we let $u^i\rightarrow \infty$ for some $1\leq i\leq n-1$. Thus the basic estimates (*) actually holds everywhere on $\mathbb{P}T_X$ provided any point $(w,[u])\in \Supp(\chi)\subset X\subset \mathbb{P}T_X $ is such that $\|w\|$ is sufficiently close to $1$.
\end{proof}

\begin{proof}[Proof of Main Theorem (Theorem \ref{main})]
We will prove the base point freeness, separation of tangents and separation of points one by one.

\subsubsection*{Base Point Freeness}
For a given point $p\in \mathbb{P}T_{\overline{X}}-\pi^{-1}(D)$, we need to construct a holomorphic section $\sigma\in H^0(\mathbb{P}T_{\overline{X}},\mathcal{O}_{\mathbb{P}T_{\overline{X}}}(\ell))$ such that $\sigma(p)\neq 0$.
Under our hypothesis of canonical radii, by Proposition \ref{keyprop1} and the construction in section \ref{constr}, base point freeness of $\mathcal{O}_{\mathbb{P}T_{\overline{X}}}(\ell)$ in $\Omega$ follows for any $\ell\geq n+2$. We may require the symmetric differential vanishes at infinity to order $\ell m$ for any $m\geq 1$.

On the other hand, consider a point $p\in \Omega'$. We have $p=(p',[\mu])$, where $p'$ and $[\mu]$ denote the coordinates corresponding to respectively the base and fibre in $\mathbb{P}T_X$.
To produce symmetric differentials on $\overline{X}$, we may assume without loss of generality that, the affine coordinates of $[\mu]$ is $0$. Thus locally $p=(p',0)\in \Omega'\subset \mathbb{P}T_X$ lies in $X\subset\mathbb{P}T_X$ where we have viewed $X$ as the zero  section of $\mathbb{P}T_X$.

By composing an automorphism if necessary, we may let $p$ lifts to the origin $0\in\mathbb{B}^n$ by lifting $X$ to the universal cover $\mathbb{B}^n$. In the fundamental domain $\mathcal{F}$ of $\Gamma$ in $\mathbb{B}^n$, there is $\rho>0$ determined by the injectivity radius $d(\Omega')$ with respect to $g$ and the size of $\Omega$ such that $B(0,1-\rho)\subset \mathcal{F}$.

In order to apply Theorem \ref{hor}, let now
\[
\rho(z)=1-|z|
\]
and consider the weight function
\[
\psi:= \chi \log |z|,
\]
where
\[
\chi=
\begin{cases}
1, \quad |z|\leq 1-2\rho\\
0, \quad |z|\geq 1-\rho
\end{cases}
\]
is a smooth cutoff function with $\Supp(\chi)\subset \{1-2\rho<|z|< 1-\rho\}$. 

Then
\begin{eqnarray*}
(*_3): 
&&\sqrt{-1}\ddbar \psi\\
&=& \sqrt{-1}\left(\chi \ddbar \log |z|+\log |z| \ddbar \chi + \partial \chi\wedge \barD\log |z|+\partial \log |z|\wedge \barD\chi \right) \\
&=& \sqrt{-1}\left[\frac{\chi}{2} \ddbar \log |z|^2+ \log |z| \left(\chi''\partial |z|\wedge \barD|z|+\chi'\ddbar|z|\right)\right. \\
&&  \displaystyle +\left.\chi'\partial |z| \wedge \frac{1}{|z|}\barD |z|+ \frac{1}{|z|}\partial |z|\wedge \chi'\barD|z| \right]\\
&=& \sqrt{-1}\left[\frac{\chi}{2} \ddbar \log |z|^2
+\chi'\log|z|\ddbar|z|
+\left(\chi''\log |z|+\frac{2\chi'}{|z|}\right)\partial |z|\wedge \barD|z|\right] \\
&=& \sqrt{-1}\left[\frac{\chi}{2} \ddbar \log |z|^2
+\chi'\log|z|\left( \frac{1}{2|z|}\ddbar|z|^2-\frac{1}{|z|}\partial |z|\wedge \barD |z| \right)\right.\\
&& \displaystyle +\left.\left(\chi''\log |z|+\frac{2\chi'}{|z|}\right)\partial |z|\wedge \barD|z|\right] \\
&=& \sqrt{-1}\left[\frac{\chi}{2} \ddbar \log |z|^2
+\frac{\chi'\log|z|}{2|z|}\ddbar|z|^2\right]\qquad (\text{positive term}). \\
&&\displaystyle +\left(-\frac{\chi'\log|z|}{|z|}+\chi''\log |z|+\frac{2\chi'}{|z|}\right)\sqrt{-1}\partial |z|\wedge \barD|z| \qquad (\text{negative term})\\
\end{eqnarray*}
By estimates similar to Equations $(*_2)$, we have
\begin{eqnarray*}
-\frac{\chi'\log|z|}{|z|}
&\geq& -(-\frac{2}{\rho})(\log(1-2\rho))(\frac{1}{1-2\rho})
=-(-\frac{2}{\rho})(-2\rho(1+o(\rho)))(\frac{1}{1-2\rho})\\
&=&-\frac{4}{1-2\rho}(1+o(\rho))=O(1),\\
\chi''\log|z|
&\geq& \frac{2}{\rho^2}(\log(1-2\rho))=\frac{2}{\rho^2}(-2\rho(1+o(\rho)))=\frac{-4}{\rho}+O(1),\\
\frac{2\chi'}{|z|}
&\geq & 2(\frac{-2}{\rho})(\frac{1}{1-2\rho})
=-4(\frac{1}{\rho}+\frac{2}{1-2\rho})=\frac{-4}{\rho}+O(1).
\end{eqnarray*}
The negative term becomes
\[
-\frac{\chi'\log|z|}{|z|}+\chi''\log |z|+\frac{2\chi'}{|z|}\geq -\frac{8}{\rho}+O(1).
\]
Thus
\begin{eqnarray*}
R.H.S.(*_3) &\geq & \sqrt{-1}\left[\frac{\chi}{2} \ddbar \log |z|^2
+\frac{\chi'\log|z|}{2|z|}\ddbar|z|^2\right]\qquad (\text{positive term}) \\
&&\displaystyle +\left(-\frac{8}{\rho}+O(1)\right)\sqrt{-1}\partial |z|\wedge \barD|z| \qquad (\text{negative term}).
\end{eqnarray*}
It follows that the negativity is of growth $1/\rho$. 
To verify the basic estimates 
\begin{equation}
\Theta(\mathcal{O}_{\mathbb{P}T_{X}}(\ell),(\hat{g}^*)^{\ell})+\Ric(\omega_{\mathbb{P}T_X})+\sqrt{-1}\ddbar\psi\geq \varepsilon \omega_{\mathbb{P}T_X},
\label{basicBPF}
\end{equation}
on $\Omega'=\mathbb{P}T_X-\Omega$, it suffices to focus on the base direction, i.e., the $z$-coordinates.
By the asymptotic behaviour (\ref{intasymp}), it follows that the curvature term is positive and of growth $1/\rho^2$ as $\rho\rightarrow 0$. Thus for $\rho$ sufficient close to $0$, the curvature terms dominate the negativity coming from the weight function term. Hence the Basic Estimates (\ref{basicBPF}) is verified whenever $\rho$ is sufficiently close to $0$.

To construct a symmetric differential non-vanishing at the point $p$, first of all, let $\chi$ be a cutoff function supported in a small neighbourhood $U$ of the interior point $p$. Let $e$ be a holomorphic basis of the line subbundle $\mathcal{O}_U(1)\subset \mathcal{O}_{\mathbb{P}T_X}(1)|_U$.
Since we have verified the Basic Estimates (\ref{basicBPF}), by Theorem \ref{hor}, a solution to the equation
\[
\barD u = \barD (\chi e^\ell)
\]
exists and satisfies the estimates
\begin{equation}
\int_{\mathbb{P}T_X} \|u\|^2e^{-(4n-2)\psi}dV
\leq \frac{1}{\varepsilon}\int_{\mathbb{P}T_X} \|\barD (\chi e^\ell)\|^2e^{-(4n-2)\psi}dV<\infty,
\label{sectestBPF}
\end{equation}
where $dV=\frac{\omega^{2n-1}_{\mathbb{P}T_X}}{(2n-1)!}$.
Let $U\cong \Delta^{2n-1}(\epsilon)$ be a small neighbourhood of $p$ for some $0<\epsilon<1$, which is identified to the origin in $\Delta^{2n-1}(\epsilon)$. In the chart $\Delta^{2n-1}(\epsilon)$, for some constant $C>0$ and $dV_e$ the Euclidean volume on $\mathbb{C}^{2n-1}$, we have
\[
e^{-(4n-2)\psi}dV=C\frac{1}{|z|^{4n-2}}dV_e=C\frac{1}{r^{4n-2}} r^{4n-3}dS=\frac{C}{r}dS,
\]
where $r=|z|$ is the polar radius and $dS$ is the volume form of the unit sphere.
The last equality combined with the estimates (\ref{sectestBPF}) shows that $u$ must vanish at the point $p$ to order at least $1$.
Let
\[
\sigma:=\chi e^\ell-u
\]
Then $\barD\sigma=\barD({\chi e^\ell})-\barD u =0$ implies that $\sigma$ is holomorphic. Moreover,
$\sigma(p)=e^\ell(p)-u(p)=e^\ell(p)\neq 0$. 
Thus we have demonstrated the base point freeness of $H^0(\mathbb{P}T_{\overline{X}},\mathcal{O}_{\mathbb{P}T_{\overline{X}}}(\ell))$ in $\Omega$ for $\ell\geq n+2$.

To conclude, we have shown that $H^0(\mathbb{P}T_{\overline{X}},\mathcal{O}_{\mathbb{P}T_{\overline{X}}}(\ell))$ is base point free on $\mathbb{P}T_{\overline{X}}-\pi^{-1}(D)$ for $\ell\geq n+2$.

\subsubsection*{Separation of Tangents}
We need to show that for any $p\in \mathbb{P}T_X-\pi^{-1}(D)$, there is a holomorphic section $\sigma\in H^0(\mathbb{P}T_{\overline{X}},\mathcal{O}_{\mathbb{P}T_{\overline{X}}}(\ell))$ such that its differential $d\sigma_p\neq 0$. The  assertion essentially follows similar arguments as the proof in the base point freeness. 
If $p\in \Omega$ is a point near infinity, then the separation of tangents essentially follows from the Lefschetz's embedding theorem for abelian variety. Here we also only need to require $\ell\geq n+2$.
If $p\in \Omega'$ is a point in the interior part, then one consider a more singular weight function $4n\psi$ and start the construction with the local holomorphic section $z_i e^\ell$ for any coordinate function $z_i$. The required injectivity radius would be larger than that in the proof of base point freeness.

\subsubsection*{Separation of Points}
We need to show that for any pair of points $p,q\in \mathbb{P}T_X-\pi^{-1}(D)$, there are holomorphic sections $\sigma,\eta\in H^0(\mathbb{P}T_{\overline{X}},\mathcal{O}_{\mathbb{P}T_{\overline{X}}}(\ell))$ such that $\sigma(p)\neq 0, \sigma(q)=0$ and $\eta(p)=0,\eta(q)\neq 0$. There are three possibilities:
\begin{enumerate}
\item $p,q\in \Omega$ are points near infinity.
\item $p\in \Omega$ and $q\in \Omega'$.
\item $p,q\in \Omega'$ are both in the interior part.
\end{enumerate}

1) Suppose $p,q\in \Omega$. If $p$ and $q$ lie in two disjoint neighbourhoods (of sizes depending on injectivity radii), then the result follows from the construction in the proof of base point freeness applied to $p$ and $q$ separately. If both $p$ and $q$ lie in the same neighbourhood, then the assertion follows from the Lefschetz's embedding theorem for abelian variety. Here we may take any $\ell\geq n+2$.

2) Suppose $p\in \Omega$ and $q\in \Omega'$. Then we only need to apply the construction in the proof of base point freeness for $p$ and $q$ separately.

3) Suppose $p,q\in \Omega'$. Without loss of generality, assume $p$ is lifted to the origin $0$ in the universal cover $\mathbb{B}^n$ and $q$ some point $a\neq 0$ in $\mathbb{B}^n$. There is $\rho>0$ depending on the injectivity radius so that $B(0,1-\rho)$ corresponds to the metric ball neighbourhood of $p\in X$ of maximal size. 

If $a\notin B(0,1-\rho)$, then we can find two disjoint neighbourhood $B_0$ and $B_a$ (whose sizes also depend on the injectivity radii) of $0$ and $a$ respectively in $\mathbb{B}^n$ and apply the construction of the base point freeness to produce the desired sections. This would require a larger injectivity radius than the one in the original construction of sections for base point freeness of only one point because the radii of the neighbourhoods of $p$ and $q$ (or $0$ and $a$) cannot be taken to be the maximal ones.

If $a\in B(0,1-\rho)$, then instead of $\psi=\chi \log|z|$, we only need to consider another weight function
\[
\psi_1:= \chi (\log |z|+ \log |z-a|).
\]
The verification of the basic estimates and construction of section then follows similarly to the case of base point freeness (which also requires a possibly larger injectivty radius since $\psi_1$ would give two negative terms of same order of magnitude instead of when considering $\psi$).
\end{proof}

\begin{proof}[Proof of Corollary \ref{main2cor}]
It suffices to show that the injectivity radius of the interior part of $\overline {X}$ and the canonical radii of the cusps both increase upon lifting along a tower of coverings of $X$.
Write $\pi_1(X_k)=\Gamma_k$. Since $\{X_k\}$ is a tower of covering of $X$, we have $\Gamma_{k+1}\subset \Gamma_k$ for any $k\geq 0$.

Let $p\in \Omega'$ be a point in the interior compact part. We want to show that the injectivity radius of $p$ goes to infinity when lifted along the tower of coverings. Let $p_0$ be any lifting of $p$ to the universal cover $\mathbb{B}^n$ of $X$. By definition, the injectivity radius $d_k=d_k(p)$ is the largest radius so that the metric ball $B_g(p_0,d_k)$ (where $g$ is the Bergman metric of $\mathbb{B}^n$) injects into $X_k$ via the covering map $\mathbb{B}^n\rightarrow X_k$. This implies that 
\[
\{\gamma\in \Gamma_k \mid \gamma (\overline{B_g(p_0,d)})\cap \overline{B_g(p_0,d)}\neq \emptyset\}=\{id\}
\]
for any $d\leq d_k$. Let $r>0$ be a real number. Since $\Gamma_k$ is discrete, 
\[
|\{\gamma\in \Gamma_k \mid \gamma (\overline{B_g(p_0,r)})\cap \overline{B_g(p_0,r)}\neq \emptyset\}|<\infty.
\]
Note that $\Gamma_k\rightarrow \{\id\}$ as $k\rightarrow \infty$, there is $k_r\geq 0$ such that
\[
\{\gamma\in \Gamma_{k_r} \mid \gamma (\overline{B_g(p_0,r)})\cap \overline{B_g(p_0,r)}\neq \emptyset\}=\{id\}
\]
for any $k\geq k_r$. It follows that $d_k\geq r$ for any $k\geq k_r$. Since $r>0$ is arbitrary, the sequence $\{d_k\}$ must be increasing and $d_k\rightarrow \infty$ as $k\rightarrow \infty$.

To see that the canonical radii increase upon lifting along the tower of coverings, consider in particular a cusp $b$ of $\Gamma=\Gamma_0$, which lifts by the covering $X_k\rightarrow X$ to a cusp $b_k$ of $\Gamma_k\subset \Gamma_0=\Gamma$. 

Let $\Omega^{(u_b}\subset \overline{X}$ be the fundamental neighbourhood of the infinity $T_b$ in $\overline{X}$, where $\Omega^{(u_b)}\rightarrow T_b$ carries the disk bundle structure so that 
\[
r_0:=e^{-\frac{2\pi}{\tau_0}u_b}
\] 
is the canonical radius of $b$. Here $u_b>0$ is the canonical height of $b$ and $\tau_0>0$ is the generator of $\Gamma_{b,0}=\Gamma_0\cap [U_b,U_b]$
(recall that $U_b=\rad P_b$ is the unipotent radical of the parabolic subgroup $P_b\subset \Aut(\mathbb{B}^n)$ and $P_b$ is the isotropy subgroup of $b$ in $\Aut(\mathbb{B}^n)$). 
Let 
\[
\Gamma_{b,k}:=\Gamma_k\cap [U_b,U_b]\subset \Gamma_0\cap[U_b,U_b].
\]
Since $\Gamma_{b,0}$ is a one-parameter group generated by $\tau_0$, the normal subgroup $\Gamma_{b,k}$ is also a one parameter and is generated by some $\tau_k\geq \tau_0$.
Since $\Gamma_k\rightarrow \{\id\}$ as $k\rightarrow \infty$, we also have $\Gamma_{b,k}\rightarrow \{\id\}$ and as $k\rightarrow \infty$. It follows that $\{\tau_k\}$ is an increasing sequence of positive real numbers such that $\tau_k\rightarrow \infty$ as $k\rightarrow \infty$.

Now we identify $\mathbb{B}^n$ to the upper half space model $S$. Recall that the fundamental neighbourhood $\Omega^{(u_b)}$ of the infinity $T_b$ in $\overline{X}$ is obtained by taking interior closure in the punctured disk bundle structure of the horoball $B(u_b)/\Gamma_b$, where the canonical height $u_b$ is the smallest positive real number so that $B(u_b)$ is invariant under $\Gamma_b=\Gamma_{b,0}$.
The fundamental neighbourhood $\Omega^{(u_{b_k})}$ of the infinity $T_{b_k}$ in $\overline{X_k}$ is obtained by taking interior closure of the puncture disk bundle structure of the horoball $B(u_{b_k})/\Gamma_{b,k}$, where $u_{b_k}$ is the smallest positive real number so that $B(u_{b_k})$ is invariant under $\Gamma_{b,k}$. Then $\Gamma_{b,k}\subset \Gamma_{b,0}$ implies that $u_{b_k}\leq u_{b}$. Since $\Gamma_{b,k}\rightarrow \{\id\}$ as $k\rightarrow 0$, it follows that $\{u_{b_k}\}$ is a decreasing sequence of positive real numbers so that $u_{b_k}\rightarrow 0$ as $k\rightarrow \infty$.

The canonical radius of $b_k$ is given by
\[
r_k=e^{\frac{-2\pi }{\tau_k}u_{b_k}}.
\]
Thus $r_k$ is an increasing sequence and $r_k\rightarrow 1$ as $k\rightarrow \infty$.
\end{proof}

\subsubsection*{Acknowledgement}
This article contains material taken from the Ph.D thesis of the author, who would like to sincerely thank his advisor Ngaiming Mok for the continued guidance and support. The author would also like to thank the referee for valuable suggestions.

\end{document}